\documentclass[11pt, reqno]{amsart}
\usepackage{mathrsfs}
\usepackage[active]{srcltx}
\usepackage{mathrsfs,amsmath}
\usepackage{mathtools}
\usepackage{longtable}
\usepackage{todonotes}
\usepackage{amssymb}
\usepackage{cite}
\allowdisplaybreaks

\usepackage{xcolor}
\definecolor{bluecite}{HTML}{0875b7}

\usepackage[unicode=true,
bookmarksopen={true},
pdffitwindow=true,
colorlinks=true,
linkcolor=bluecite,
citecolor=bluecite,
urlcolor=bluecite,
hyperfootnotes=false,
pdfstartview={FitH},
pdfpagemode= UseNone]{hyperref}

\newcommand{\ler}[1]{\left(#1\right)}
\newcommand{\dd}{\mathrm{d}}

\newcommand{\cW}{\mathcal{W}}
\newcommand{\HH}{\mathbb{H}}

\newcommand{\be}{\begin{equation}}
\newcommand{\ee}{\end{equation}}

\newcommand{\abs}[1]{\left|#1\right|}
\newcommand{\norm}[1]{\left|\left|#1\right|\right|}

\newcommand{\Wog}{\mathcal{W}_1(\mathbf{G})}

\newcommand{\Wphn}{\mathcal{W}_p(\HH^n)}
\newcommand{\Wohn}{\mathcal{W}_1(\HH^n)}
\newcommand{\dwo}{d_1}

\newcommand{\Wox}{\mathcal{W}_1(X)}

\newcommand{\supp}{\mathrm{supp}}

\newcommand{\isom}{\mathrm{Isom}}

\newcommand{\bG}{\mathbf{G}}
\newcommand{\R}{\mathbb{R}}

\usepackage[left=3.2cm,right=3.2cm,top=3.2cm,bottom=3.2cm]{geometry}

\newtheorem{theorem}{Theorem}[section]
\newtheorem{lemma}{Lemma}[section]
\newtheorem{corollary}{Corollary}[section]
\newtheorem{definition}{Definition}[section]
\newtheorem{remark}{Remark}[section]
\numberwithin{equation}{section}

\usepackage{bbm}  
\usepackage{dsfont}


\title[Isometric rigidity of $\mathcal{W}_1(\mathbf{G})$ over Carnot groups]{Isometric rigidity of the Wasserstein space $\mathcal{W}_1(\mathbf{G})$ over Carnot groups}
\author[Zolt\'an M. Balogh]{Zolt\'an M. Balogh}
\address{Zolt\'an M. Balogh, Universit\"at Bern\\ Mathematisches Institut (MAI)\\ Sidlerstrasse 12\\ 3012 Bern\\ Schweiz}
\email{zoltan.balogh@unibe.ch}

\author[Tam\'as Titkos]{Tam\'as Titkos}
\address{Tam\'as Titkos, Corvinus University of Budapest, Department of Mathematics \\
Fővám tér 13-15 \\ Budapest 1093 \\ Hungary\\ and HUN-REN Alfr\'ed R\'enyi Institute of Mathematics\\ Re\'altanoda u. 13-15.\\
Budapest 1053\\ Hungary}
\email{titkos.tamas@renyi.hu}

\author[D\'aniel Virosztek]{D\'aniel Virosztek}
\address{D\'aniel Virosztek, HUN-REN Alfr\'ed R\'enyi Institute of Mathematics\\ Re\'altanoda u. 13-15.\\
Budapest 1053\\ Hungary}
\email{virosztek.daniel@renyi.hu}

\thanks{Z. M. Balogh is supported by the Swiss National Science Foundation, Grant Nr. {200020\_191978}.  }

 \thanks{T. Titkos is supported by the Hungarian National Research, Development and Innovation Office (NKFIH) under grant agreements no. K134944 and no. Excellence\_151232, and by the Momentum program of the Hungarian Academy of Sciences under grant agreement no. LP2021-15/2021.}

\thanks{D. Virosztek is supported by the Momentum program of the Hungarian Academy of Sciences under grant agreement no. LP2021-15/2021, by the Hungarian National Research, Development and Innovation Office (NKFIH) under grant agreement no. Excellence\_151232, and partially supported by the ERC Synergy Grant No. 810115.}

\thanks{Email address of the corresponding author (D\'aniel Virosztek): virosztek.daniel@renyi.hu}

\begin{document}
	\begin{abstract}
 This paper aims to study isometries of the $1$-Wasser\-stein space $\mathcal{W}_1(\mathbf{G})$ over Carnot groups endowed with horizontally strictly convex norms. Well-known examples of horizontally strictly convex norms on Carnot groups are the Heisenberg group $\HH^n$ endowed with the Heisenberg-Kor\'anyi norm, or with the Naor-Lee norm; and $H$-type Iwasawa groups endowed with a Kor\'anyi-type norm. We prove that on a general Carnot group there always exists a horizontally strictly convex norm. The main result of the paper says that if $(\mathbf{G},N_{\mathbf{G}})$ is a Carnot group where $N_{\mathbf{G}}$ is a horizontally strictly convex norm on $\mathbf{G}$, then the Wasserstein space $\mathcal{W}_1(\mathbf{G})$ is isometrically rigid. That is, for every isometry $\Phi:\mathcal{W}_1(\mathbf{G})\to\mathcal{W}_1(\mathbf{G})$ there exists an isometry $\psi:\bG\to\bG$ such that $\Phi=\psi_{\#}$.
 \end{abstract}
	\maketitle 
	\tableofcontents
	
	%	\vspace{-1cm}
%%%%%%%%%%%%%%%%%%%%%%%%%%%%%%%%%%%%%%%%%%%%%%%%%%%%%%%%%%%%%%%%%%%%%%
%%%%%%%%%%%%%%%%%%%%%%%%%%%%%%%%%%%%%%%%%%%%%%%%%%%%%%%%%%%%%%%%%%%%%%
%\newpage
 \section{Introduction: motivation and main results}\label{s1:intro}
As Hermann Weyl wrote in \cite{Weyl}, whenever a structure-endowed entity is given, one can expect a deep insight into the constitution of that entity by determining its group of automorphisms. In this spirit, there has been recent activity to understand the structure of isometries of various metric spaces of measures, see e.g. \cite{DKM, dolinar-molnar, kuiper, geher-titkos, molnar-levy}. One of the most intensively studied metric structures of measures nowadays is the so-called $p$-Wasserstein space $\mathcal{W}_p(X)$, where the underlying space $(X,\varrho)$ is a complete and separable metric space (see the precise definition later). The study of the geometric properties of these metric spaces is motivated by the theory of optimal mass transportation. Indeed, connections between the geometry of $X$ and $\mathcal{W}_2(X)$ have been investigated by Lott and Villani in the groundbreaking paper \cite{LV}. The pioneering papers of Bertrand and Kloeckner \cite{K, BK, BK2, K2} started to explore fundamental geometric features of $2$-Wasserstein spaces, including the description of complete geodesics and geodesic rays, determining their different type of ranks, and understanding the structure of their isometry group $\isom\big(\mathcal{W}_2(X)\big)$. 
Isometries and isometric embeddings of $p$-Wasserstein spaces have been intensively studied in recent years for various underlying spaces $X$ by Bertrand, Gehér, Kiss, Kloeckner, Santos-Rodriguez, and the authors; see \cite{BTV, BK, BK2, isemb-jmaa, GTV1, GTV2, TnSn, exotic, KT, K, K2, S-R}. 

An important feature of $p$-Wasserstein spaces is that the underlying space $X$ embeds isometrically into $\mathcal{W}_p(X)$, and similarly, the isometry group $\isom(X)$ of the underlying space $X$ embeds into $\isom\big(\mathcal{W}_p(X)\big)$ by push-forward of measures by isometries. In many cases, it is known that this embedding is surjective, i.e., $\isom(X)$ and $\isom\big(\mathcal{W}_p(X)\big)$ are isomorphic. In such a case we call the Wasserstein space isometrically rigid. The value of $p$ can play a distinctive role here. For example, in \cite{GTV1} it has been shown that $\mathcal{W}_p([0,1])$ is isometrically rigid if and only if $p\neq 1,$ and the results of \cite{K} and \cite{GTV1} together imply that $\mathcal{W}_p(\R)$ is isometrically rigid if and only if $p\neq 2.$ These examples show that the isometric rigidity of $\mathcal{W}_1(X)$ does not imply the rigidity of $\mathcal{W}_{p}(X)$ for $p>1,$ and the reverse implication is also not true in general. 
\par 
In our recent manuscript \cite{BTV}, we studied isometries and isometric embeddings of $\Wphn$, the $p$-Wasserstein space over the Heisenberg group endowed with the Heisenberg-Korányi distance. We found that $\Wphn$ is isometrically rigid if $p>1$. Our method was based on a description of complete geodesics and geodesic rays (isometric embeddings of $\mathbb{R}$ and $\mathbb{R_+}$ into $\Wphn$), which is not available if $p=1$. Our aim is to overcome this difficulty and find a method that can be applied not only in the case of the Heisenberg group but to a much larger class of spaces. The main result of this paper is 
the following.

\begin{theorem}\label{maintheorem}
Let $(\bG,N_{\bG})$ be a Carnot group where $N_{\bG}$ is a horizontally strictly convex norm on $\bG$. Then the Wasserstein space $\mathcal{W}_1(\bG)$ is isometrically rigid. That is, for every isometry $\Phi:\Wog\to\Wog$ there exists an isometry $\psi:\bG\to\bG$ such that $\Phi=\psi_{\#}$.
\end{theorem}
The definition of horizontally strictly convex norms was introduced in the setting of the Heisenberg group in \cite{BFS} in connection to isometric embeddings. This concept can easily be generalized to the setting of Carnot groups and it is done in the next section. We note here in advance that the horizontal strict convexity of a norm is a stronger property than the usual strict convexity, although the terminology may suggest it is weaker. Indeed, horizontal strict convexity of a norm means that the subadditivity of the norm is saturated only if the vectors involved are scalar multiples of each other \emph{and horizontal} --- see Definition \ref{def:hscn}. As an immediate consequence of this theorem, we get that the $1$-Wasserstein space over the Heisenberg group endowed by the Heisenberg-Kor\'anyi norm or the Naor-Lee norm is isometrically rigid. Here, we use the fact that these norms are horizontally strictly convex, as proven in \cite{BFS}.   Other examples where our theorem applies are $H$-type Iwasawa groups endowed with a Kor\'anyi type norm (see \cite{H-type}). 
The relevance of our result is underlined by the following existence theorem of horizontally strictly convex norms on general Carnot groups:
\begin{theorem}\label{existence-shc-norms}
Let $(\bG, \cdot)$ be a general Carnot group. Then there exists a horizontally strictly convex  norm $N_{\bG}: \bG \to [0, \infty)$. 
\end{theorem}
The paper is organized as follows: in Section 2 we fix notation and recall preliminaries on Wasserstein spaces and on Carnot groups. In Section 3 we prove the main result: Theorem \ref{maintheorem}, and in Section 4 we prove Theorem \ref{existence-shc-norms}.

 \section{Preliminaries on Wasserstein spaces and Carnot groups}
 
 We start with notations that will be used in the sequel, for more details and references on $p$-Wasserstein spaces we refer the reader to any of the following comprehensive textbooks \cite{AG, Figalli, Santambrogio, Villani, V}. In this paper, we will focus on the case $p=1$.
 
Let $(X,\varrho)$ be a complete and separable metric space. We denote by $\mathcal{M}^+(X)$ the set of all non-negative finite Borel measures on $X$. For all $r>0$ the symbol $\mathcal{M}_{\varrho}^r(X)$ stands for the set
\begin{equation*}
    \mathcal{M}_{\varrho}^r(X)= \Big\{ \mu\in\mathcal{M}^+(X)\,\Big|\,\mu(X)=r,~\int_X\varrho(x,\widehat{x})~\mathrm{d}\mu(x)<\infty~\mbox{for some}~\widehat{x}\in X.\Big\}
\end{equation*}
\begin{definition}
A non-negative Borel measure $\Pi$ on $X \times X$ is called a coupling for $\mu,\nu\in\mathcal{M}_{\varrho}^r(X)$ if the marginals of $\Pi$ are $\mu$ and $\nu$, that is, $$\Pi\ler{A \times X}=\mu(A)\qquad\mbox{and}\qquad\Pi\ler{X \times B}=\nu(B)$$ for all Borel sets $A,B\subseteq X$. The set of all couplings is denoted by $C(\mu,\nu)$. 
\end{definition}
For two non-negative finite measures with the same total weight $\mu,\nu\in\mathcal{M}_{\varrho}^r(X)$ we can define the distance function $d_1^{(r)}:\mathcal{M}_{\varrho}^r(X)\times\mathcal{M}_{\varrho}^r(X)\to\mathbb{R}_{\geq0}$ as follows:
\begin{equation}
 \label{eq:wasser_def}
d_1^{(r)}\ler{\mu, \nu}:=\inf_{\Pi \in C(\mu, \nu)} \iint_{X \times X} \varrho(x,y)~\dd \Pi(x,y).
\end{equation}

\begin{definition}
The $1$-Wasserstein space $\mathcal{W}_1(X)$ is the set $\mathcal{M}_{\varrho}^1(X)$ endowed with the so-called $1$-Wasserstein distance $d_1^{(1)}$.
\end{definition}

It is known (see e.g. Theorem 1.5 in \cite{AG} with $c=\varrho$) that the infimum in \eqref{eq:wasser_def} is, in fact, a minimum in this setting. Those couplings that minimize \eqref{eq:wasser_def} are called optimal couplings, or optimal transport plans. 

From now on, with a slight abuse of notation, we will write $d_1$ instead of $d_1^{(r)}$, no matter what the value of $r$ is. However, it is important to emphasize that if we write $d_1(\mu,\nu)$ for some non-negative finite measures $\mu$ and $\nu$, then it is always assumed that $\mu(X)=\nu(X)$. The advantage of this notation will be clear later in \eqref{transinv}.

Next, we highlight two special features of the $1$-Wasserstein metric. The first special feature is that there is a duality phenomenon, the so-called Kantorovich-Rubinstein duality. For given $\mu,\nu\in\Wox$ the $1$-Wasserstein distance can be computed as a supremum
 \begin{equation}\label{KR}
     d_1(\mu,\nu)=\sup_{\|f\|_L=1}\int_{X}f~d(\nu-\mu)
 \end{equation}
where $\|f\|_L$ denotes the Lipschitz norm of $f$, i.e. the smallest constant $C$ such that 
$$|f(q)-f(q')| \leq C\varrho(q,q') \qquad\mbox{for all}\quad q,q'\in X.$$
The second special feature is a consequence of the duality: the $1$-Wasserstein distance is translation invariant in the following sense:
\begin{equation}\label{transinv}
    d_1(\mu+\xi,\nu+\xi)=d_1(\mu,\nu).
\end{equation}
For more details see Theorem 4.4 in \cite{Edwards}.

As the terminology suggests, all notions introduced above are strongly related to the theory of optimal transportation. Indeed, for given sets $A$ and $B$ the quantity $\Pi(A\times B)$ is the weight of mass that is transported from $A$ to $B$ as $\mu$ is transported to $\nu$ along the transport plan $\Pi$.

The symbol $\supp\mu$ stands for the support of $\mu\in\mathcal{P}(X)$. The set of Dirac measures will be denoted by $\Delta(X)=\{\delta_x\,|\,x\in X\}$. For $\mu,\nu\in\Wox$ and $\lambda\in(0,1)$ let us denote $M_{\lambda}(\mu,\nu)$ the metric $\lambda$-ratio set
$$M_{\lambda}(\mu,\nu)=\big\{\xi\in\Wox\,\big|\,d_1(\mu,\xi)=\lambda d_1(\mu,\nu),~~d_1(\xi,\nu)=(1-\lambda) d_1(\mu,\nu)\big\}.$$

\begin{definition} Given two metric spaces $(Y,d_Y)$ and $(Z,d_Z)$, a map $f:Y\to Z$ is an isometric embedding if $d_Z(f(y),f(y'))=d_Y(y,y')$ for all $y,y'\in Y$. A self-map $\psi\colon Y\to Y$ is called an isometry if it is a surjective isometric embedding of $Y$ onto itself. The symbol $\isom(\cdot)$ will stand for the group of all isometries.
\end{definition}
For an isometry $\psi\in\isom(X)$ the induced push-forward map is
$$\psi_\# \colon \mathcal{P}(X)\to\mathcal{P}(X);\qquad\big(\psi_\#(\mu)\big)(A)=\mu(\psi^{-1}[A])$$ 
for all Borel sets $A\subseteq X$ and $\mu\in\mathcal{P}(X)$, where 
$\psi^{-1}[A]=\{x\in X\,|\, \psi(x)\in A\}.$ We call $\psi_\#(\mu)$ the \emph{push-forward} of $\mu$ with $\psi$.

It is important to remark that $\mathcal{W}_1(X)$ contains an isometric copy of $X$. Indeed, since $C(\delta_q,\delta_{q'})$ has only one element (the Dirac measure $\delta_{(q,q')})$, we have that \begin{equation*}d_1(\delta_q,\delta_{q'})=\iint_{X\times X}\varrho(u,v)~\mathrm{d}\delta_{(q,q')}(u,v)=\varrho(q,q'),
\end{equation*}
and thus the embedding
\begin{equation}
X\hookrightarrow\mathcal{W}_p(X);\qquad x\mapsto\delta_x
\end{equation}
is distance preserving. Furthermore, the set of finitely supported measures (in other words, the collection of all finite convex combinations of Dirac measures)
\begin{equation*}
    \mathcal{F}(X)=\Bigg\{\sum_{j=1}^k\lambda_j\delta_{x_j}\,\Bigg|\,k\in\mathbb{N},x_j\in X,\,\lambda_j\geq0\,(1\leq j\leq k),\,\sum_{j=1}^k\lambda_j=1\Bigg\}
\end{equation*}
is dense in $\mathcal{W}_1(X)$, see e.g. Example 6.3 and Theorem 6.18 in \cite{Villani}. Another important feature is that isometries of $X$ appear in $\isom(\mathcal{W}_p(X))$ by means of a natural group homomorphism.
\begin{definition} Let $\#$ denote the map which embeds $\isom(X)$ into $\isom(\mathcal{W}_p(X))$
\begin{equation} 
 \label{eq:hashtag}
\#\colon \, \mathrm{Isom} (X) \rightarrow \mathrm{Isom}  \ler{\mathcal{W}_p(X)}, \qquad \psi \mapsto \psi_\#.
\end{equation}
Isometries that belong to the image of $\#$ are called trivial isometries. If $\#$ surjective, i.e., if every isometry is trivial, then we say that $\mathcal{W}_p(X)$ is isometrically rigid.
\end{definition}
\begin{definition}
A geodesic segment (or shortly: geodesic) is a curve $\gamma: [a,b] \to \mathcal{W}_1(X)$ such that 
$$d_1(\gamma(t),\gamma(s)) = C|t-s|$$ for all $t,s \in [a,b]$. Note, that by reparametrising the curve $\gamma$ we can always achieve that $C=1$. Geodesics with $C=1$ will be called unit-speed geodesics.
\end{definition}

In the main result of this paper, the underlying metric space $X$ will be a Carnot group. There are several recent studies on the metric  structure  of these groups (see e.g. \cite{BTW, BKS, L}), as 
they serve as natural examples of sub-Riemannian geometries where the theory of optimal mass transport was recently developed \cite{BR,FR,R}.  In this paper, we refer to \cite{BLU} for background and we shall be using the same notations. A Carnot group $(\mathbf{G}, \cdot)$ of step $r$ is a connected,
simply connected, nilpotent Lie group whose Lie algebra $\mathfrak{g}$
of left-invariant vector fields admits a stratification, i.e. there
exist non-zero subspaces $\{V_j\}_1^r$ such that
$$
\mathfrak{g}=\bigoplus_{i=1}^r V_i,\qquad  [V_1,V_j]=V_{j+1}\not=0
\quad \texttt{\rm for}\  j=1,\ldots r-1,\qquad [V_1,V_r]=0.$$
We assume that a scalar product is given on $\mathfrak{g}$ for which the
subspaces $V_j$ are mutually orthogonal. The first layer $V_1$ of
the Lie algebra plays a key role: its elements are called {horizontal vectors}.
 We denote by $n_i$ the dimension of the vector space $V_i$.  Let $X=\{X_1,X_2,\ldots,X_n\}$
be an orthonormal basis of  $\mathfrak{g}$ such that for every $j$ (with
$1\le j\le r)$  the set
$$
\{X_{i};\ \texttt{\rm with}\ n_0+\ldots +n_{j-1}<i\le n_1+\ldots
+n_{j}\}
$$
is a basis for $V_j$ (here we put $n_0=0$ and $N_j=n_0+\ldots + n_j$
and hence $N_r=n$). 

The exponential map $\exp: \mathfrak{g}\to \mathbf{G}$ is defined by
$$\exp(x_1X_1+\ldots+x_nX_n)=(x_1,\ldots,x_n),$$ and it is
 a global diffeomorphism; we denote by
$\xi=(\xi_1,\xi_2,\ldots,\xi_r)$ the inverse of $\exp,$ where
$\xi_j:\bG\to V_j.$ This mapping defines the so-called coordinates of the first kind on $\bG$, 
that is identified as a point set with $\R^n$. The group $\bG$ is somewhat similar to the usual 
Euclidean space in the sense that it has a natural family of non--isotropic dilations and 
and the group law has polynomial expressions.
To be more precise, we note that the natural family of non-isotropic dilations for $\lambda>0$ on
$\mathfrak{g}$ associated with its grading is given by
$$\tilde\delta_\lambda(v_1+v_2+\ldots+v_r)=\lambda v_1+\lambda^2
v_2+\ldots+\lambda^r v_r,$$ if $v_i\in V_i,$ $1\le i\le r.$ By means of
the exponential map, one lifts these dilations to the family of the
automorphisms of $\bG\simeq \R^n $ $\delta_{\lambda}(x)=\exp( \tilde\delta_\lambda(
\xi(x))),$ i.e.
$$
\delta_{\lambda}(x)=(\lambda^{\sigma_1}x_1,\ldots,\lambda^{\sigma_n}x_n),
$$
with $x\in\R^n$ and $\lambda>0,$ where $\sigma_i=j$ for
$N_{j-1}<i\le N_j.$ 

The Baker--Campbell--Dynkin--Hausdorff formula gives the relation between the  bracket relations in the Lie algebra $\mathfrak{g}$ and  the expressions of  the group operations on
$\mathbf{G}.$ Taking into consideration the symmetries encoded in the bracket relations in $\mathfrak{g}$, the group law will also possess some useful symmetry relations (see \cite{HS}): According to this, we can write an element in $\mathbf{G}$ as  $p = (x,y)$ where $x\in \R^{n_1} $ and $y \in \R^{n_2 + \ldots + n_s}$. In this notation the group operation in $\mathbf{G}$ can be written as 
$$ (x_1,y_1) \cdot (x_2, y_2) = (x_1 +x_2, y_1 + y_2 + P_1(x_1, x_2) + P_2((x_1, y_1), (x_2, y_2)),$$
where $P_1: \R^{n_1} \times \R^{n_1} \to \R^{n_2+ \ldots + n_s}$ and $P_2: \R^{n} \times \R^n \to \R^{n_2+ \ldots + n_s}$ are mappings with polynomial components in the variables $x_1, x_2$ resp. $(x_1,y_1), (x_2, y_2)$. We
denote by $e=0\in\R^n$ the null element in $\bG$ and note that the inverse of the element $q=(x_1,\dots,x_n)$ is simply $q^{-1}=-q=(-x_1, \dots, -x_n).$
\par

\color{black}

In what follows we turn to the metric structure of our Carnot group. We consider left-invariant homogeneous metrics on $\mathbf{G}$ defined as follows. The first step is to define a norm on $\mathbf{G}$ that is compatible with the group's law:

\begin{definition} \label{norm}
Let $(\bG,\cdot)$ be a Carnot group with neutral element $e$. We say that a map $N: \bG \to\mathbb{R}_{\geq0}$ is a \emph{norm} on $\bG$ if it satisfies
\begin{align}
&i)\ N(q)=0\Leftrightarrow q=e,\ \forall q\in \bG,\nonumber\\
&ii)\ N(q^{-1})=N(q),\ \forall q\in \bG,\nonumber\\
&iii)\ N(q\cdot q')\leq N(q)+N(q'),\ \forall q,q'\in \bG.\nonumber
\end{align}
\end{definition}

\begin{definition}
Let $(\bG,\  \cdot)$ be a Carnot group. A metric $d:\bG\times \bG\to\mathbb{R}_{\geq0}$
is called \emph{left-invariant}, if for every $q_{o}\in \bG$, the map $L_{q_{o}}:(\bG,d)\to (\bG,d),\ q\mapsto q_{o}\cdot q$ is an isometry, that is, $d(q_{o}\cdot q,q_{o}\cdot q')=d(q,q')$, for all $q,q'\in \bG$.
\end{definition}
%We note that the map $L_{g_{o}}$ is always bijective with inverse
%$L_{g_{o}^{-1}}$.
%On the other hand,
Every norm $N:\bG\to\mathbb{R}_{\geq0}$ induces
a left-invariant metric $d_{N}:\bG\times \bG\to\mathbb{R}_{\geq0}$, by the formula $d_N(q,q') := N(q^{-1}\cdot q')$.
Conversely,  given a left-invariant metric $d: \bG \times \bG \to \R_{+}$, the map $N(q)= d(q, e)$ defines a norm on $\bG$. 

\begin{definition}
A norm $N:\bG \to\mathbb{R}_{\geq0}$ on a Carnot group is called \emph{homogeneous} if
\begin{align}
\ N(\delta_{\lambda}(p))=\lambda N(p),\ \text{for all }\lambda>0,\ \text{for all } p\in\mathbf{G}.\nonumber
\end{align}
\nonumber
\end{definition}

It is easy to see that a norm $N$ on $\bG$ is homogeneous if and only if its associated left-invariant metric is homogeneous in the sense that $d_N(\delta_{\lambda}(p),\delta_{\lambda}(q))=\lambda d_N (p,q)$. %According to this definition, every "homogeneous left-invariant metric on $\mathbb{H}^{n}$ " is also a "homogeneous quasi-distance on $\mathbb{H}^{n}$ ", in the sense of [Le Donne-Rigot, Def.2.20].
Every left-invariant distance on $\bG$ induced by a homogeneous norm is a homogeneous distance.
From now on, we will use the expression \emph{homogeneous distance on $\bG$} to talk about the left-invariant metric induced by a homogeneous norm. The topology induced by any homogeneous distance on $\bG$ coincides with the Euclidean topology on the base space $\R^n$.

In fact, once the homogeneous distances are known to be continuous with respect to the standard topology on 
$\R^n$ one can show by a standard argument the even stronger fact that all such norms are bi-Lipschitz equivalent. On the other hand, the metric structure induced by a homogeneous norm $N$ on $\bG$ is very different from $\R^n$ endowed with the Euclidean distance $d_{E}$. The two distances $d_N$ and $d_{E}$ are not bi-Lipschitz equivalent for any choice of homogeneous norm $N$ on $\bG$. For a comparison of the Hausdorff dimension of sets with respect to these two norms we refer to \cite{BTW}.

The simplest non-trivial Carnot group of step 2 is the Heisenberg group $\mathbb{H}^n \simeq \R^{2n+1} $
where the group operation is given by 
$$ q\cdot q'  = (x,y,z)\cdot (x',y', z')= \Big(x+x', y+y',
z+z'+2\sum_{i=1}^n (x'_iy_i-x_iy'_i)\Big), $$   and the action of the non-isotropic dilations is  defined as 
$\delta_{\lambda}:\mathbb{ H}^n
\to \mathbb{H}^n$, $\lambda>0$ $$\delta_{\lambda}(x,y,z)= (\lambda x, \lambda y,\lambda^2 z).$$

The left-invariant, homogeneous Heisenberg-Kor\'anyi metric $d_H$ is
defined by the homogeneous Kor\'anyi  norm on $\mathbb{H}^n$  
$$\|(x,y,z)\|_K =
\Big(\big(\sum_{i=1}^n(x^2_i+y^2_i)\big)^2 + z^2\Big)^{\frac{1}{4}}.  $$
  Another interesting norm on $\mathbb{H}^n$ is the norm $N_{LN}: \mathbb{H}^n \to \R_{+}$ introduced by Lee and Naor in \cite{LN} and is given by: 
 $$ N_{LN}(x,y,z)= \sqrt{||(x,y,z)||_K^2 + ||(x,y)||_E^2},$$
 where $||\cdot||_E$ is the usual Euclidean norm. We leave it as an exercise for the interested reader to verify that the above formula satisfies the conditions of Definition \ref{norm}. We refer to \cite{BFS} for more details, and many additional examples of homogeneous norms defined on the Heisenberg group. In that paper, the class of {\it horizontally strictly convex} norms has been introduced in the setting of the Heisenberg group. We can define this property in more general Carnot groups as follows:

\begin{definition} \label{def:hscn}
A homogeneous norm $N: \bG \to \R_{+}$ is horizontally strictly convex if and only if
$$q,q' \in \bG , q, q' \neq e, \ \ N(q\cdot q') = N(q) +N(q') \ \ \implies$$
 $q$ and $q'$ lie on the same horizontal line through the origin, i.e., there exists $z\in \R^{n_1}$  such that 
 $$ q=(sz,0), q'=(s'z, 0), \ \ \text{ for some}  \ s, s'\in \R.$$

\end{definition}

It is easy to check that a homogenous norm $N$ on $\bG$ is horizontally strictly convex if and only if for all $q_1, q_2, q \in \bG$ with $q_1 \neq q$ and $q_2 \neq q$ the equality 
$$ d_N(q_1, q_2) = d_N(q, q_1) + d_N(q, q_2) $$
implies that $q_1, q_2$ belong to the same horizontal line through $q$:
$$ l= \{ q\cdot (sz, 0): s \in \R \}, $$
for some non-zero $z\in \R^{n_1}$.

It was shown in \cite{BFS} that the Kor\'anyi norm and the Lee-Naor norm in the Heisenberg group are horizontally strictly convex. On more general H-type Iwasava groups of step two one can still define a norm of Kor\'anyi type that is horizontally strictly convex \cite{H-type}.
On the other hand, we note, that many homogeneous norms on Carnot groups are {\it not} horizontally strictly convex. A prominent example is the Carnot-Carath\'eodory norm coming from the sub-Riemannian Carnot-Carath\'odory metric.  Here {\it any pair of points} can be joined by a geodesic and therefore there are many points saturating the triangle inequality for any given pair of points. The same situation occurs in the case of more general sub-Finsler metrics \cite{BFS}. 

 In the following, we consider another class of norms on the Heisenberg group $\HH^n$ that are \emph{not} horizontally strictly convex. Let $1\leq p \leq \infty$ and $0<a\leq n^{-1/2},$ and consider the function

$$
N_{p,a}: \, \HH^n \rightarrow [0, \infty); \, (x,y,z) \mapsto \max \left\{ \norm{(x,y)}_{p}, a \sqrt{\abs{z}}\right\} 
$$
where $\norm{.}_{p}$ denotes the $l^p$ norm on $\R^{2n}.$ Then $N_{p,a}$ is a homogeneous norm on $\HH^n$ --- see \cite[Example 5.4]{BFS}. Now we check that for any $1 \leq p \leq \infty$ this norm is not horizontally strictly convex.
Indeed, for $|z| < 1/a$ consider the two points $q=(1,0,\ldots ,0,z), \, q'=(1,0,\ldots ,0, -z) \in \HH^n.$  Then
$$
2= N_{p,a}\ler{q \cdot q'}=N_{p,a}\ler{q}+N_{p,a}\ler{q'}
$$
but clearly $q$ and $q'$ do not lie on the same horizontal line, and thus $N_{p,a}$ is not horizontally strictly convex.
\par
\color{black} 

Motivated by the statement of Theorem \ref{maintheorem} we intend to find horizontally strictly convex norms in general Carnot groups. In this general setting, this could be a non-trivial matter. However, a prominent candidate seems to be the class of norms introduced by Hebisch and Sikora in \cite{HS}.  Let us recall, that the  Hebisch-Sikora norm on $N_{HS}: \bG \to \R_{+}$ is defined as:
\begin{equation} \label{HS}
N_{HS}(q) := \inf \{ t>0: \delta_{\frac{1}{t}}  (q) \in B(0,r) \},
\end{equation}
where $B(0,r) \subseteq \R^n$ is the usual Euclidean ball in $\R^n$ centered at the origin, with radius $r$. 

In \cite{HS} Hebisch and Sikora proved, that there exists $r_0>0$ such that for all $0<r<r_0$ the function $N_{HS}$ from \eqref{HS} is a homogeneous norm. One can verify by a direct calculation that the above Lee-Naor norm is a particular instance of the Hebisch-Sikora norms in the setting of the Heisenberg group.

 The Hebisch-Sikora norm $N_{HS}$ was already proven to be useful in several studies: Lee and Naor \cite{LN} used $N_{LN}$ on the Heisenberg group to give a counterexample to the Goemans-Linial conjecture.  Le Donne and Rigot proved that in the setting of the Heisenberg group (see \cite{LR1}) and also in the setting of Carnot groups of step two (see \cite{LR2}) the Hebisch-Sikora norm $N_{HS}$ satisfies the so-called Besicovitch covering property used in differentiation of measures. In a recent paper \cite{L} Li used the norm $N_{HS}$ in connection to the analyst's traveling salesman problem in Carnot groups. As we shall see in Theorem \ref{T:HSCN}: the Hebisch-Sikora norm is horizontally strictly convex, verifying the statement of Theorem \ref{existence-shc-norms}.

%%%%%%%%%%%%%%%%%%%%%%%%%%%%%%%%%%%%%%%%%%%%%%%%%%%%%%%%%%%%%%%%%%%%%%
%%%%%%%%%%%%%%%%%%%%%%%%%%%%%%%%%%%%%%%%%%%%%%%%%%%%%%%%%%%%%%%%%%%%%%

\section{Isometric rigidity of $\Wog$: Proof of Theorem \ref{maintheorem}}
First, let us introduce some notations. Let $(X,\varrho)$ be a general metric space.  For $q,q'\in X$ the symbol $[q,q']$ stands for the metric segment between $q$ and $q'$
$$[q,q']=\{r\in X\,|\,\varrho(q,r)+\varrho(r,q')=\varrho(q,q')\}.$$
The property that $[q,q']=\{q,q'\}$ i.e. that the interior of the metric segment $[q,q']$ is empty will be important in the sequel. In fact, in the proof of Theorem \ref{maintheorem} we need an even stronger property: we will use the symbol $q\sim q'$ for the relation
\begin{equation} \label{eq:sim-def}
    q\sim q' \quad\Longleftrightarrow\quad  
     \begin{cases}
       \forall r\in X:~~\varrho(q,q')=\varrho(q,r)+\varrho(r,q')~~\Longrightarrow~~ r\in\{q,q'\} \\
       \forall r\in X:~~\varrho(r,q)=\varrho(r,q')+\varrho(q',q)~~\Longrightarrow~~ r\in\{q,q'\}\\
       \forall r\in X:~~\varrho(q',r)=\varrho(q',q)+\varrho(q,r) ~~\Longrightarrow~~ r\in\{q,q'\}\\
     \end{cases}
\end{equation}
Observe that if $X$ has at least two points then $q\sim q'$ implies $q\neq q'$. And obviously, $q\sim q'$ implies $[q,q']=\{q,q'\}$.

Our aim in the next Lemma is to characterize those pairs of measures such that there is only one unit-speed geodesic connecting them in $\Wox$, namely the linear interpolation. Although this might be known, or even folklore, we did not find it in the literature, so we include it for the sake of completeness.  
\begin{lemma}\label{KeyLemma}
Let $(X,\varrho)$ be a complete and separable metric space, and let $\mu$ and $\nu$ be two different elements of $\Wox$ with $d_1(\mu,\nu):=T>0$. Then the following statements are equivalent.
\begin{itemize}
\item[(i)] The signed measure $\mu-\nu$ can be written as 
    \begin{equation*}
    \mu-\nu=c\delta_q-c\delta_{q'}
    \end{equation*}
    for some positive real number $c$ and $q,q'\in X$ such that $[q,q']=\{q,q'\}$. In other words, $\mu$ and $\nu$ can be written as
    \begin{equation*}
\mu=\eta+c\delta_q\qquad\mbox{and}\qquad\nu=\eta+c\delta_{q'}
    \end{equation*}
    for some non-negative measure $\eta$, $c>0$, and $q,q'\in X$ such that $[q,q']=\{q,q'\}$.  \item[(ii)] For every $\lambda\in(0,1)$ the metric $\lambda$-ratio set is a singleton, that is,
    $$M_{\lambda}(\mu,\nu)=\{(1-\lambda)\mu+\lambda\nu\}.$$
    \item[(iii)] There is only one unit-speed geodesic $\gamma:[0,T]\to\Wox$ connecting $\mu$ and $\nu$ ($\gamma(0)=\mu$, $\gamma(T)=\nu$). That geodesic is the linear interpolation
    \begin{equation}\label{EuclideanAverage}
     \gamma(t)=\Bigg(1-\frac{t}{T}\Bigg)\mu+\frac{t}{T}\nu\qquad(t\in[0,T]).
    \end{equation}

\end{itemize}
Moreover, if (i) holds with the stronger assumption $q\sim q'$, then the unique unit-speed geodesic segment between $\mu$ and $\nu$ cannot be extended in any direction if and only if $\eta=0$.
\end{lemma}
\begin{proof}
    To prove (i)$\Longrightarrow$(ii) assume that $\mu=\eta+c\delta_q$, $\nu=\eta+c\delta_{q'}$, and $[q,q']=\{q,q'\}$. If $f$ is any $1$-Lipschitz function (i.e., $|f(v)-f(u)|\leq\varrho(v,u)$ for all $u,v\in X$) then for any $\Pi\in C(\mu,\nu)$ we have
\begin{equation*}
\begin{split}
    \iint_{X\times X} f(v)-f(u)~\mathrm{d}\Pi(u,v)&\leq\iint_{X\times X} |f(v)-f(u)|~\mathrm{d}\Pi(u,v)\\
    &\leq\iint_{X\times X} \varrho(v,u)~\mathrm{d}\Pi(u,v).
    \end{split}
\end{equation*}
The left-hand side of the above inequality can be written as
\begin{equation}
\begin{split}
\iint_{X\times X} f(v)-f(u)~\mathrm{d}\Pi(u,v)&=\int_X f(v)~\mathrm{d}\nu(v)-\int_X f(u)~\mathrm{d}\mu(u)\\
&=\int_X f~d(\nu-\mu),
\end{split}
\end{equation}
so we see that if a function $f$ solves the dual side and $\Pi$ solves the primal side of the Kantorovich-Rubinstein duality then
\begin{equation}\label{supp1}
    f(v)-f(u)=\varrho(v,u)\qquad\mbox{for}~\Pi\mbox{-almost every pair}~(u,v).
\end{equation}
According to the assumption, $\nu-\mu=c(\delta_{q'}-\delta_q)$, so $f_q(z):=\varrho(q,z)$ solves the dual problem \ref{KR}. Indeed, $\int_X f \dd (\nu-\mu)=c\ler{f(q')-f(q)}$ for any (Borel measurable) function $f: X \to \R,$ and if $f$ is furthermore $1$-Lipschitz, that is, $\norm{f}_{L}=1,$ then $c\ler{f(q')-f(q)} \leq c \varrho(q,q').$ On the other hand, for the above-defined function $f_q$ we clearly have $\int_X f_q \dd (\nu-\mu)=c \varrho(q,q').$ Consequently, for any optimal coupling $\Pi\in C(\mu,\nu)$ we know that if  $(w,q')\in\supp\Pi$ then
\begin{equation}\label{supp2}
   \varrho(q',w)=f_q(q')-f_q(w)=\varrho(q,q')-\varrho(q,w),
\end{equation}
or equivalently, $\varrho(q,q')=\varrho(q,w)+\varrho(q',w)$. Since $[q,q']=\{q,q'\}$, this implies that $w\in\{q,q'\}$. The assumption $\nu-\mu=c\ler{\delta_{q'}-\delta_{q}}$ implies that the mass to be transported from $X\setminus \{q'\}$ to $\{q'\}$ by the optimal transport plan $\Pi$ is at least $c.$ That is, $\Pi\ler{\ler{X\setminus \{q'\}} \times \{q'\}} \geq c,$ which in turn implies that $\Pi\ler{\{q\} \times \{q'\}}=\Pi\ler{\{(q,q')\}} \geq c.$
%In other words, $\Pi(q,q')=c$, 
However, by translation invariance we have $d_1(\mu, \nu) = c \varrho(q,q')$ and hence 
\begin{equation} \label{eq:transl-inv-cons}
c \varrho(q,q')=
d_1(\mu,\nu)=\iint_{X\times X}\varrho(u,v)~\mathrm{d}\Pi(u,v)\geq \Pi\ler{\{(q,q')\}} \varrho(q,q')
\end{equation}
which implies that $\Pi\ler{\{(q,q')\}} \leq c$ and consequently $\Pi\ler{\{(q,q')\}}=c.$ Moreover, \eqref{eq:transl-inv-cons} implies that 
$$
\iint_{\ler{X \times X} \setminus \{(q,q')\}} \varrho(u,v)~\mathrm{d}\Pi(u,v)=0. 
$$
Therefore, we conclude that there is only one optimal transport plan: the one which moves mass of magnitude $c$ from $q$ to $q'$ and leaves everything else intact. In mathematical terms, the only optimal coupling $\Pi_*$ is given by
\begin{equation}\label{Pistar}
\Pi_*:=c\delta_{(q,q')}+(\mathrm{Id}\times \mathrm{Id})_{\#}\eta.
\end{equation}

Now fix a $\lambda\in(0,1)$ and take a $\xi\in M_{\lambda}(\mu,\nu)$. We have to show that $\xi=(1-\lambda)\mu+\lambda\nu$.
Since $d_1(\mu,\xi)+d_1(\xi,\nu)=d_1(\mu,\nu)$, we know on the one hand that 
if $\Pi_{\mu,\xi}\in C(\mu,\xi)$ and $\Pi_{\xi,\nu}\in C(\xi,\nu)$ are optimal couplings, then gluing them together we get an optimal coupling between $\mu$ and $\nu$, and that optimal coupling must be $\Pi_*$. On the other hand,
\begin{equation}
\begin{split}
        d_1(\mu,\nu)&=\int_X f_q(z)~d(\nu-\mu)(z)\\
        &=\int_X f_q(z)~d([\nu-\xi]+[\xi-\mu])(z)\\
        &=\int_X f_q(z)~d(\nu-\xi)(z)+\int_X f_q(z)~d(\xi-\mu)(z)\\
        &\leq d_1(\nu,\xi)+d_1(\xi,\nu)\\
        &=d_1(\mu,\nu),
\end{split}
\end{equation}
which means that $f_q(z)$ solves the dual problem for both pairs $(\nu,\xi)$ and $(\xi,\mu)$. The same calculation as in \eqref{supp1}-\eqref{supp2} ensures that if $(w,q')\in\supp\Pi_{\xi,\nu}$ then $w\in\{q,q'\}$. Combining these two facts we get that $\xi$ can be written for some $0\leq\alpha\leq c$ as \begin{equation}\label{xi}
\xi=\eta+\alpha\delta_q+(c-\alpha)\delta_{q'}.
\end{equation}
Again, using the translation invariance of the distance, we have

\begin{equation}
\begin{split}
d_1(\mu,\xi)&=d_1(\eta+c\delta_q,\eta+\alpha\delta_q+(c-\alpha)\delta_{q'})\\
&=d_1(c\delta_q,\alpha\delta_q+(c-\alpha)\delta_{q'})=(c-\alpha)\varrho(q,q').
\end{split}
\end{equation}
This implies that
\begin{equation}(1-\lambda)c\varrho(q,q')=(1-\lambda)d_1(\mu,\nu)=(c-\alpha)\varrho(q,q'),
\end{equation}
or equivalently, $\alpha=\lambda c$. Now we see from \eqref{xi} that
$$\xi=\eta+\lambda c\delta_q+(c-\lambda c)\delta_{q'}=(1-\lambda)\mu+\lambda\nu.$$\\

To prove (ii)$\Longrightarrow$(iii), assume by contradiction that there exist at least two different unit-speed geodesics $\gamma_1$ and $\gamma_2$, connecting $\mu$ and $\nu$. Since the two curves are not identical, there exists at least one $t\in(0,T)$ such that $\gamma_1(t)\neq\gamma_2(t)$. For that $t\in(0,T)$ we have $\{\gamma_1(t),\gamma_2(t)\}\subseteq M_{1-\frac{t}{T}}(\mu,\nu)$, a contradiction.\\

To prove (iii)$\Longrightarrow$(i), assume by contradiction  that (i) does not hold. This is the case if either (a) $\mu-\nu=c\delta_q -c\delta_{q'}$ for some $q,q'\in X$ and $c>0$, but $[q,q']\neq\{q,q'\}$; or (b) $\mu-\nu\neq c\delta_q-c\delta_{q'}$ for any $q,q'\in X$ and $c>0$.

Let us handle case (a) first. Since $\mu-\nu=c\delta_q -c\delta_{q'}$ by assumption, we can write $\mu$ and $\nu$ as $\mu=\eta+c\delta_q$ and $\nu=\eta+c\delta_{q'}$ for some non-negative measure $\eta$ and $c>0$. On the other hand, since $[q,q']\neq\{q,q'\}$, we can choose a $z\notin\{q,q'\}$ such that $\varrho(q,q')=\varrho(q,z)+\varrho(z,q')$. Using this $z\in X$ we can construct a unit-speed geodesic $\widetilde{\gamma}:[0,T]\to\Wox$ between $\mu$ and $\nu$, passing through $\xi:=\eta+c\delta_z$.  Set $T'=d_1(\mu,\xi)$ and consider
\[   
\widetilde{\gamma}(t):= 
     \begin{cases}
       \eta+\Big(1-\frac{t}{T'}\Big)c\delta_q+\frac{t}{T'}c\delta_z &\quad\text{if}\quad t\in[0,T'], \\
       \eta+\Big(1-\frac{t-T'}{T-T'}\Big)c\delta_z+\frac{t-T'}{T-T'}c\delta_{q'}&\quad\text{if}\quad t\in(T',T]. \\ 
     \end{cases}
\]
An elementary calculation shows that $\widetilde{\gamma}$ is a unit-speed geodesic between $\mu$ and $\nu$ (such that $\widetilde{\gamma}(T')=\xi$), which is of course different from the linear interpolation \eqref{EuclideanAverage}.

To handle case (b) assume that the signed measure $\mu-\nu$ is not of the form $c\delta_q-c\delta_{q'}$ for any $c>0$ and $q,q'\in X$. Let us denote by $\mu'=(\mu-\nu)_+$ and $\nu'=(\mu-\nu)_-$ its positive and negative part, respectively. Note, that since $(\mu-\nu)(X)=0$ we have that $\mu'(X)= \nu'(X)$. According to the assumption, $\supp\mu'\cap\supp\nu'=\emptyset$ and at least one of them (by symmetry, let's say $\mu'$) is not a constant multiple of a Dirac measure. Then we can write $\mu'$ as a sum $\mu'=\mu'_1+\mu'_2$ such that $\mu'_1$ is not a constant multiple of $\mu'_2$. Indeed, choose a set $S$ such that $0<\mu'(S)<\mu'(X)$ and define $\mu'_1$ and $\mu'_2$ as follows: $\mu'_1(A):=\mu'(A\cap S)$  and $\mu'_2(A):=\mu'(A\setminus S)$ for all Borel sets $A\subseteq X$. This decomposition of $\mu'$ also induces a decomposition $\nu'=\nu'_1+\nu'_2$. To see this, let us choose an optimal coupling $\Pi\in C(\mu',\nu')$ and set $\nu'_1(B)=\Pi(S\times B)$, $\nu'_2(B):=\Pi((X\setminus S)\times B)$ for all Borel sets $B\subseteq X$. These decompositions guarantee that we can construct at least two different unit-speed geodesics between $\mu$ and $\nu$. (In words, we can transport the whole $\mu'_1$ first, and then $\mu'_2$, or vice versa.) Set $T'=d_1(\mu'_1,\nu'_1)$, and observe that
$$T=d_1(\mu,\nu)=d_1(\mu'_1,\nu'_1)+d_1(\mu'_2,\nu'_2)$$
according to \eqref{transinv} and the optimality of $\Pi$. Now we leave it to the interested reader, to verify that the curves $\gamma_1$ and $\gamma_2$ defined by 

\[   
\gamma_1(t)= 
     \begin{cases}
       \eta+\mu'_2+\big(1-\frac{t}{T'}\big)\mu'_1+\frac{t}{T'}\nu'_1 &\quad\text{if}\quad t\in[0,T'], \\
       \eta+\nu_1'+\big(1-\frac{t-T'}{T-T'}\big)\mu'_2+\frac{t-T'}{T-T'}\nu'_2 &\quad\text{if}\quad t\in(T',T] \\ 
     \end{cases}
\]
and
\[   
\gamma_2(t)= 
     \begin{cases}
       \eta+\mu'_1+\big(1-\frac{t}{T-T'}\big)\mu'_2+\frac{t}{T-T'}\nu'_2 &\quad\text{if}\quad t\in[0,T-T'], \\
       \eta+\nu_2'+\big(1-\frac{t-(T-T')}{T'}\big)\mu'_1+\frac{t-(T-T')}{T'}\nu'_1 &\quad\text{if}\quad t\in(T-T',T] \\ 
     \end{cases}
\]
are two different unit-speed geodesics connecting $\mu$ and $\nu$, a contradiction. This finishes the proof of (i)$\Longleftrightarrow$(ii)$\Longleftrightarrow$(iii).

Finally, in order to prove the second part of the statement, we have to show that if (i) holds with the stronger assumption $q\sim q'$, then the unique unit-speed geodesic segment between $\mu$ and $\nu$ cannot be extended in any direction if and only if $\eta=0$.

Assume first by contradiction that $\mu=\eta+c\delta_q$ and $\nu=\eta+c\delta_{q'}$ where $q\sim q'$, but $\eta\neq0$. Then at least one of $\mu$ and $\nu$ is not a Dirac measure. By symmetry, we can assume that $\nu$ is not a Dirac measure. Then the geodesic can be extended towards $\delta_{q'}$. Indeed, set $T'=d_1(\mu,\delta_{q'})$, and observe that the curve
\[   
\widetilde{\gamma}(t)= 
     \begin{cases}
       \eta+c\Big(\big(1-\frac{t}{T}\big)\delta_q+\frac{t}{T}\delta_{q'}\Big) &\quad\text{if}\quad t\in[0,T], \\
       c\delta_{q'}+\Big(1-\frac{t-T}{T'-T}\Big)\eta+\frac{t-T}{T'-T}(1-c)\delta_{q'} &\quad
       \text{if}\quad t\in(T,T'] \\ 
     \end{cases}
\]
is a unit-speed geodesic segment connecting $\mu$ and $\delta_{q'}$, such that $\widetilde{\gamma}(T)=\nu$, a contradiction. To see the converse direction, assume that $\eta=0$, or equivalently, $\mu=\delta_q$ and $\nu=\delta_{q'}$ where $q\sim q'$. Then we have $T=d_1(\mu,\nu)=d_1(\delta_q,\delta_{q'})=\varrho(q,q')$, and the unique geodesic segment between $\delta_q$ and $\delta_{q'}$ is
\begin{equation}
    \gamma:[0,T]\to\mathcal{W}_1(X),\qquad\gamma(t)=\Big(1-\frac{t}{T}\Big)\delta_q+\frac{t}{T}\delta_{q'}. 
    \end{equation}
We have to show that $\gamma$ cannot be extended in any direction. That is, for all $\varepsilon>0$ there is no unit-speed geodesic segment $\gamma_1:[-\varepsilon,T]\to\Wohn$ such that $\gamma_1|_{[0,T]}=\gamma$, and similarly, there is no geodesic segment $\gamma_2:[0,T+\varepsilon]\to\Wohn$ such that $\gamma_2|_{[0,T]}=\gamma$. We present the argument for $\gamma_1$, the case of $\gamma_2$ is similar. Set $\xi:=\gamma_1(-\varepsilon)$ and recall that the triple $\gamma_1(-\varepsilon)=\xi,\gamma_1(0)=\delta_q, \gamma_1(T)=\delta_{q'}$ satisfies the triangle inequality with equality:
\begin{equation}
d_1(\xi,\delta_{q'})=d_1(\xi,\delta_q)+d_1(\delta_q,\delta_{q'}).
\end{equation}
Using that $1=\int_{X}1~\mathrm{d}\xi(r)$, we have $d_1(\delta_q,\delta_{q'})=\varrho(q,q')=\int_{X}\varrho(q,q')\mathrm{d}\xi(r)$, and thus we can write the above equality as
\begin{equation}
\int_{X}\varrho(r,q')~\mathrm{d}\xi(r)=\int_{X}\varrho(r,q)~\mathrm{d}\xi(r)+\int_{X}\varrho(q,q')~\mathrm{d}\xi(r),
\end{equation}
or equivalently,
\begin{equation}
    \int_X \ler{\varrho(r,q)+\varrho(q,q')-\varrho(r,q')} \dd \xi(r)=0.
\end{equation}
As $\varrho(r,q)+\varrho(q,q')-\varrho(r,q') \geq 0$ by the triangle inequality, this means that for $\xi$-almost every $r\in X$ we have
\begin{equation}
\varrho(r,q')=\varrho(r,q)+\varrho(q,q').
\end{equation}
According to $q\sim q'$ this implies that $\supp(\xi)\subseteq\{q,q'\}$, and thus $\gamma_1(-\varepsilon)=\gamma_1(t)$ for some $t\in[0,T]$, a contradiction.
\end{proof}
Now we are ready to prove the main result of the paper.\\

\noindent\textbf{Proof of Theorem 1.1.} Let $G$ be a Carnot group with a horizontally strictly convex norm $N_{\bG}$, and denote by $d$ the metric induced by $N_{\bG}$. We want to prove that the Wasserstein space $\Wog$ over $(\bG,d)$ is isometrically rigid. The first step is to show that the $\Phi$-image of pairs of the form $\{\delta_q,\delta_{q'}\}$ where $q\sim q'$ is again a pair of Dirac measures.

We note here that for elements of a Carnot group with a horizontally strictly convex norm, the relation $q \sim q'$ defined by \eqref{eq:sim-def} is equivalent to the a priori weaker condition $[q,q']=\{q,q'\}.$ Moreover, these equivalent conditions have clear and transparent characterization: $q \sim q'$ if and only if $[q,q']=\{q,q'\}$ if and only if $q^{-1}\cdot q'$ is \emph{not} horizontal. Indeed, if $q^{-1}\cdot q'$ is horizontal, then
$$
\{q,q'\} \subsetneq \left\{ q \cdot \delta_{\alpha}(q^{-1} \cdot q') \, \middle| \, \alpha \in [0,1]\right\} \subseteq [q,q']
$$
as
$$
d(q,q \cdot \delta_{\alpha}(q^{-1} \cdot q'))=\alpha N_{\bG}(q^{-1} \cdot q')
$$
and
$$
d(q',q \cdot \delta_{\alpha}(q^{-1} \cdot q'))=(1-\alpha) N_{\bG}(q^{-1} \cdot q'),
$$
and consequently, $q \nsim q'.$ On the other hand, if $q \nsim q'$ then there exists an $r \notin \{q,q'\}$ such that $r, q$ and $q'$ saturate the triangle inequality --- that is, one of the three equations in \eqref{eq:sim-def} holds true. We can assume without too much loss of generality that
$
d(r, q')=d(r,q)+d(q,q')
$
as the two other cases are very similar. In this case we get that $N_{\bG}(r^{-1} \cdot q)+N_{\bG}(q^{-1} \cdot q')=N_{\bG}(r^{-1} \cdot q).$ By the horizontal strict convexity of $N_{\bG},$ see Definition \ref{def:hscn}, we immediately get that $r^{-1}\cdot q$ and $q^{-1} \cdot q'$ lie on the same horizontal line through the origin, in particular, $q^{-1} \cdot q'$ is a horizontal vector.

\noindent\underline{Step 1.} Let us take a pair $\{\delta_q,\delta_{q'}\}$ where $q\sim q'$, and set $T:=d(q,q')>0$. According to Lemma \ref{KeyLemma}, there is only one unit-speed geodesic $\gamma:[0,T]\to\Wog$ connecting $\delta_q$ and $\delta_{q'}$, and this geodesic cannot be extended in any directions. Since an isometry is a bijection, we obtain that there is only one unit-speed geodesic between $\Phi(\delta_q)$ and $\Phi(\delta_{q'})$ as well. Again, using Lemma \ref{KeyLemma}, this means that there exists a non-negative measure $\eta$, a constant $0<c\leq 1$ and two points $r,r'\in X$ satisfying $[r,r']=\{r,r'\}$ such that $\Phi(\delta_q)=\eta+c\delta_r$ and $\Phi(\delta_{q'})=\eta+c\delta_{r'}$. In fact, since the norm is horizontally strictly convex, we also have $r\sim r'$, and thus we know that $\eta=0$. Indeed, assume indirectly that $\eta\neq 0$. Then the geodesic segment between $\Phi(\delta_q)$ and $\Phi(\delta_{q'})$ could be extended according to the second part of Lemma \ref{KeyLemma}, and the pre-image of this geodesic segment would be an extension of $\gamma$, a contradiction.

\noindent\underline{Step 2.} In the second step we show that there exists an isometry $\psi:\bG\to\bG$ such that $\Phi(\delta_q)=\delta_{\psi(q)}$ for all $q\in\bG$. Since $N_{\bG}$ is horizontally strictly convex, with a non-horizontal perturbation for any $q\in \bG$ we can find a $q'\in \bG$ such that $q\sim q'$.
Since $\Phi$ is bijective, this implies that $\Phi$ maps the set of Dirac masses onto itself. Indeed, it is clear by Step 1 that $\Phi\ler{\delta_q}$ is a Dirac mass for any $q \in \bG,$ and for any $r \in \bG$ there is a $\mu \in \cW_1(\bG)$ such that $\Phi(\mu)=\delta_r$ as $\Phi$ is surjective. Moreover, $\mu$ is a Dirac mass because $\mu=\Phi^{-1}\ler{\delta_r},$ and $\Phi^{-1}$ is an isometry of $\cW_1(\bG)$ which sends Dirac masses to Dirac masses by Step 1. Setting $\psi$ as $\Phi(\delta_q):=\delta_{\psi(q)}$ we obtain a map $\psi:\bG\to\bG$, which is in fact an isometry of $\bG$. Indeed, the distance preserving property follows from the fact that $\dwo(\delta_q,\delta_r)=d(q,r)$ for all $q,r\in\bG,$ and the surjectivity of $\psi$ follows from the fact that $\Phi$ maps the set of Dirac masses onto itself.

\noindent\underline{Step3.} Since $\Phi=\psi_{\#}$ if and only if $\Phi\circ\ler{\psi_{\#}}^{-1}= \mathrm{Id},$ we can assume without loss of generality that $\psi(q)=q$ for all $q\in\bG$. Our aim under this assumption will be to show that $\Phi(\mu)=\mu$ for a dense subset of $\Wog$. As $\Phi$ is continuous, this will finish the proof. Let us introduce the notation \begin{equation}\label{Fnull}
\mathcal{F}_{\sim}(\bG):=\Big\{\sum_{j=1}^n\lambda_j\delta_{q_j}\,\Big|\,n\in\mathbb{N},\,q_j\sim q_l\,\mbox{for all}\,1\leq j<l\leq n\Big\}
\end{equation}
and show that $\mathcal{F}_{\sim}(\bG)$ is dense in $\Wog$. Here again, we can take advantage of the horizontal strict convexity of the norm. Namely that for all finite subsets $\{q_i\}_{i\in I}\subseteq \bG$ ($|I|<\infty$) and for all $\varepsilon>0$ with small non-horizontal perturbations taken step-by-step we can find a finite subset $\{q_i^{\varepsilon}\}_{i\in I}\subseteq \bG$ such that
\begin{equation}\label{Prop2a}
d(q_i,q_i^{\varepsilon})<\varepsilon\quad\mbox{for all}\quad i\in I,
\end{equation}
and
\begin{equation}\label{Prop2b}
q_i^{\varepsilon}\sim q_j^{\varepsilon}\quad\mbox{for all}\quad i,j\in I.
\end{equation}
 Indeed, these assumptions imply that for any $\widetilde{\varepsilon}>0$ and for any $\xi=\sum_{j=1}^n\lambda_j\delta_{q_j}$ we can find an $\varepsilon>0$ and $\{q_i^{\varepsilon}\}_{i=1}^n$ satisfying \eqref{Prop2a} and \eqref{Prop2b} such that the measure $\xi':=\sum_{j=1}^n\lambda_j\delta_{q_j^{\varepsilon}}$ belongs to $\mathcal{F}_{\sim}(\bG)$ and $\dwo(\xi,\xi')<\widetilde{\varepsilon}$. Therefore, $\mathcal{F}_{\sim}(\bG)$ is dense in $\mathcal{F}(\bG)$ and thus in $\Wog$.

What remains to be proven is that $\Phi$ fixes all elements of $\mathcal{F}_{\sim}(\bG)$. 
We already know that $\Phi$ fixes all Dirac measures. Fix a $k\in\mathbb{N}$ and suppose we proved the statement for all elements of $\mathcal{F}_{\sim}(\bG)$ that are supported on at most $k$ points. Take a measure $\xi=\sum_{j=1}^{k+1}\lambda_j\delta_{q_j}\in\mathcal{F}_{\sim}(\bG)$ supported on $k+1$ points. We can write $\xi$ as $$\xi=(1-\lambda)\cdot\xi_1+\lambda\cdot\xi_2$$ with some $0<\lambda<1$ and $\xi_1,\xi_2\in\mathcal{F}_{\sim}(\bG)$ whose supports are sets with $k$ elements. Indeed, $$\xi_1:=\sum_{j=1}^{k-1}\lambda_j\delta_{q_j}+(\lambda_k+\lambda_{k+1})\delta_{q_k},\quad\xi_2:=\sum_{j=1}^{k-1}\lambda_j\delta_{q_j}+(\lambda_k+\lambda_{k+1})\delta_{q_{k+1}},$$ and $\lambda:=\frac{\lambda_{k+1}}{\lambda_k+\lambda_{k+1}}.$ Observe that $\xi_1$ and $\xi_2$ satisfies the conditions of Lemma \ref{KeyLemma} (i) with 
\begin{equation*}\eta=\sum_{j=1}^{k-1}\lambda_j\delta_{q_j},\quad c=\lambda_k+\lambda_{k+1},\quad q=q_k,\quad q'=q_{k+1}.\end{equation*}
Using Lemma \ref{KeyLemma} we know on the one hand that $M_{\lambda}(\xi_1,\xi_2)=\{\xi\}$, on the other hand the set $M_{\lambda}(\Phi(\xi_1),\Phi(\xi_2))$ is just the $\Phi$-image of $M_{\lambda}(\xi_1,\xi_2)$. Combining these observations with $\Phi(\xi_1)=\xi_1$ and $\Phi(\xi_2)=\xi_2$ we get $\{\xi\}=M_{\lambda}(\xi_1,\xi_2)=M_{\lambda}(\Phi(\xi_1),\Phi(\xi_2))=\{\Phi\ler{\xi}\}$. The proof is complete.\\

We close this section with a brief remark on the proof.

\begin{remark}
The scheme of the proof of Theorem \ref{maintheorem} works in the following abstract setting as well. Assume that $(X,\varrho)$ is metric space where (1) for every $q\in X$ there exists a $q'\in X$ such that $q\sim q'$, and (2) for all finite subsets $\{q_i\}_{i\in I}\subseteq X$ ($|I|<\infty$) and for all $\varepsilon>0$ there exists a $\{q_i^{\varepsilon}\}_{i\in I}\subseteq X$ such that (2/a) $\varrho(q_i,q_i^{\varepsilon})<\varepsilon$ for all $i\in I$, and (2/b) $q_i^{\varepsilon}\sim q_j^{\varepsilon}$ for all $i,j\in I$. A prominent example of such a metric space is the unit sphere $\mathbb{S}^n$ endowed with the distance inherited from the ambient Euclidean space $\R^{n+1}$ (we note that the isometric rigidity of the \emph{quadratic} Wasserstein space $\mathcal{W}_2\ler{\mathbb{S}^n}$ has recently been proven in \cite{GHTV-spheres-JMAA} for any finite dimension $n.$) Indeed, we have $x \sim x'$ for any distinct points $x, x' \in \mathbb{S}^n.$  However, it is important to remark that these conditions are only sufficient, and are very far from being necessary. For example, the Wasserstein space $\mathcal{W}_1(\R)$ is isometrically rigid, but there are no points $q,q'\in\mathbb{R}$ such that $q\sim q'$.
\end{remark}

\section{Horizontal strict convexity of the Hebisch-Sikora norm}

In this section, we prove Theorem \ref{existence-shc-norms} by verifying the horizontally strict convexity property of the Hebisch-Sikora norm. We state this result as the following:

\begin{theorem}\label{T:HSCN}
Let $(\bG, \cdot)$ be a Carnot group. Then there exists $r_0>0$ such that for all $0<r<r_0$ the function $N_{HS}$ defined by \eqref{HS} is a horizontally strictly convex norm. 
\end{theorem} 
 
\begin{proof}
Let us note, that the homogeneity of $d_{HS}$ as well as properties i) and ii) in Definition \ref{norm} follow immediately from the definition of $d_{HS}$. The triangle inequality has been proven in \cite{HS}, and our non-trivial task here is to characterize the case of equality in the triangle inequality iii)
\begin{equation} \label{triangle}
N_{HS}(q_1\cdot q_2) \leq N_{HS}(q_1) + N_{HS}(q_2) ,
\end{equation}
in order to prove the horizontal strict convexity of this norm. This needs a refined analysis. 
Although our proof is an adaptation of the main idea of the proof of Theorem 2 from \cite{HS}, we need to make a more careful analysis such that in addition to the proof of the triangle inequality we can study the equality case in \eqref{triangle} to prove horizontal strict convexity.

Let us notice first that by the homogeneity of $N_{HS}$ \eqref{triangle} can be written equivalently as 
$$N_{HS}\Big(\delta_{\frac{1}{N_{HS}(q_1) + N_{HS} (q_2)}}(q_1 \cdot q_2)\Big) \leq 1.$$
Using the relation $\delta_{t_1t_2} = \delta_{t_1} \delta_{t_2}$ we can write the above inequality as 
$$N_{HS}\Big(\delta_{\frac{N_{HS}(q_1)}{N_{HS}(q_1) + N_{HS} (q_2)}}\big(\delta_{\frac{1}{N_{HS}(q_1)}}(q_1)\big) \cdot \delta_{\frac{N_{HS}(q_2)}{N_{HS}(q_1) + N_{HS} (q_2)}}\big(\delta_{\frac{1}{N_{HS}(q_2)}}(q_2)\big) \Big)\leq 1.$$

Denoting by $p_1= \delta_{\frac{1}{N_{HS}(q_1)}}(q_1)$ and $p_2= \delta_{\frac{1}{N_{HS}(q_2)}}(q_2)$ and 
$t= \frac{N_{HS}(q_1)}{N_{HS}(q_1) + N_{HS} (q_2)}$
we see that $t\in (0,1)$ and 
\begin{equation} \label{boundary}
N_{HS}(p_1) = N_{HS}(p_2) = 1
\end{equation}
and the above inequality is written equivalently as 
\begin{equation} \label{equiv}
N_{HS}(\delta_t(p_1) \cdot \delta_{1-t}(p_2)) \leq 1. 
\end{equation}
In conclusion we see that \eqref{triangle} becomes equivalent to \eqref{equiv} subjected to \eqref{boundary}. Moreover, the equality case in \eqref{triangle} becomes equivalent to the equality case in \eqref{equiv} (subject to \eqref{boundary}).

Let us notice that by definition the unit ball centered in $0\in \bG$ with respect to $d_{HS}$ coincides as a set with the Euclidean ball $B(0,r) \subset \R^n$ and also, that  the unit sphere with respect to $d_{HS}$ coincides with 
$\partial B(0,r)$. By this observation \eqref{equiv} (subjected to \eqref{boundary}) can be formulated equivalently as follows: We have to prove that there exists $r_0>0$ such that for $0<r<r_0$, it follows that if $p_1, p_2 \in \R^n$ satisfy the property that $||p_1||^2 = ||p_2||^2 =  r^2$ then 
\begin{equation} \label{euclidean}
|| \delta_t(p_1)\cdot \delta_{1-t}(p_2) ||^2 \leq r^2,
\end{equation}
for any $t\in (0,1)$. Moreover, we have to show that equality in \eqref{euclidean} can happen if and only if 
\begin{equation} \label{HC}
p_1= p_2 = (x,0) \ \ \text{for some} \ \ x\in \R^{n_1}, ||x||^2 = r^2.
\end{equation}
It is clear that the inequality \eqref{euclidean} is saturated if $p_1$ and $p_2$ are given as in \eqref{HC}. From now on, we work on proving that equality in \eqref{euclidean} can happen only if \eqref{HC} holds.
Here, and in the following we use the same notation $|| \cdot||$ for the Euclidean norm of vectors in various Euclidean spaces of (possibly) different dimensions. 

To continue the proof let us recall from the introduction that we can write an element $p \in \bG$ as $p = (x,y)$ where $x\in \R^{n_1} $ and $y \in \R^{n_2 + \ldots + n_s}$. In this notation the group operation in $\bG$ can be written as 
$$ (x_1,y_1) \cdot (x_2, y_2) = (x_1 +x_2, y_1 + y_2 + P_1(x_1, x_2) + P_2((x_1, y_1), (x_2, y_2)),
$$
where $P_1: \R^{n_1} \times \R^{n_1} \to \R^{n_2+ \ldots + n_s}$ and $P_2: \R^{n} \times \R^n \to \R^{n_2+ \ldots + n_s}$ are mappings with polynomial components in the variables $x_1, x_2$ resp. $(x_1,y_1), (x_2, y_2)$. 
We recall from \cite{HS} that this representation uses the correspondence between the bracket relations in the Lie algebra and the group operation due to the Baker–Campbell–Dynkin–Hausdorff formula. Accordingly, $P_1$ corresponds to the contribution of vector fields in the first layer of the Lie algebra and $P_2$ corresponds to those vector fields with at least one of them coming from higher layers. Based on this, we have two crucial properties (see \cite{HS}) of $P_1$ and $P_2$ that we are going to use in the following: 
There exists a constant $C_1>0$ such that 
\begin{equation} \label{P1}
||P_1(x_1, x_2)|| \leq C_1 ||x_1|| \cdot ||x_2|| \cdot \Big\| \frac{x_1}{||x_1||}- \frac{x_2}{||x_2||}\Big\|
\end{equation}
for all $x_1, x_2 \in \R^{n_1}$ with $||x_1|| \leq 1$ and $||x_2|| \leq 1$. 
There exists a constant $C_2>0$ such that 
\begin{equation} \label{P2}
||P_2(x_1,y_1), (x_2, y_2))|| \leq C_2 (||x_1||\cdot ||y_2|| + ||x_2||\cdot ||y_1|| + ||y_1||\cdot ||y_2|| ),
\end{equation}
for all $(x_1, y_1), (x_2, y_2) \in \R^n$ with $||(x_1,y_1)|| \leq 1 $ and $||(x_2, y_2)|| \leq 1$. 

The following Euclidean norm identity for non-zero vectors $u$ and $v$ will also be used in the course of the proof:
\begin{equation} \label{identity}
|| u+ v||^2 = (||u|| + ||v||)^2 - ||u||\cdot ||v|| \cdot \Big\| \frac{u}{||u||}- \frac{v}{||v||}\Big\|^2.
\end{equation}
After this preparation we are ready to address the proof of \eqref{euclidean}.  Namely, we shall prove that if $r>0$ is small enough, then for given 
$p_1= (x_1, y_1) $ and $p_2= (x_2, y_2)$ satisfying 
$$ ||x_1||^2 +||y_1||^2 = ||x_2||^2 + ||y_2||^2 = r^2$$
it follows that 
\begin{equation} \label{new-euclidean}
|| \delta_t(x_1, y_1) \cdot \delta_{1-t}(x_2, y_2) ||^2 \leq r^2 .
\end{equation}
Moreover, we shall show that equality in \eqref{new-euclidean} holds if and only if $p_1=p_2= (x,0)$ for some 
$x\in \R^{n_1}$ with $||x||^2 = r^2$.

Using the expression of the group law with  $P_1$ and $P_2$ we can write \eqref{new-euclidean} in the form 
$$ || (t x_1 + (1-t)x_2, \delta_t(y_1) + \delta_{1-t}(y_2) + P_2(\delta_t(x_1, y_1), \delta_{1-t}(x_2, y_2)) +P_1(tx_1, (1-t)x_2)||^2 \leq r^2.$$ 
Introducing the notation 
$$ w= \delta_t(y_1) + \delta_{1-t}(y_2) + P_2(\delta_t(x_1, y_1), \delta_{1-t}(x_2, y_2)),$$
the above relation can be written as 
\begin{equation} \label{expansion}
 ||tx_1 + (1-t)x_2||^2 + ||P_1(tx_1, (1-t)x_2)||^2 +2\langle w, P_1(tx_1, (1-t)x_2)\rangle + ||w||^2 \leq r^2,
 \end{equation} 
where the notation $\langle\cdot , \cdot \rangle$ stands for the usual scalar product. 
By Young's inequality we have 
$$2\langle w, P_1(tx_1, (1-t)x_2)\rangle= t(1-t) 2 \langle w, P_1(x_1, x_2)\rangle \leq t(1-t)(||w||^2 + ||P_1(x_1, x_2)||^2).$$

Therefore the inequality \eqref{expansion} ( and thus  \eqref{new-euclidean})  will follow from 
\begin{equation} \label{readytosplit}
||tx_1 + (1-t)x_2||^2 + t(1 -t) ( 1+ t(1-t)) ||P_1(x_1, x_2)||^2 + (1 + t(1-t)) ||w||^2 \leq r^2.
\end{equation}
We are now ready to split \eqref{readytosplit} into two distinct parts:
\begin{equation}\label{split1}
\begin{split}
||tx_1 + (1-t)x_2||^2 + t(1 -t) ( 1+ t(1-t)) ||P_1(x_1, x_2)||^2 &\leq (t||x_1|| + (1-t)||x_2||)^2 \\
&\leq t||x_1||^2 + (1-t) ||x_2||^2
\end{split}
\end{equation}

and 
\begin{equation}\label{split2}
(1 + t(1-t)) ||w||^2 \leq (t\left|\left|y_1\right|\right| + (1-t)||y_2||)^2 \leq t\left|\left|y_1\right|\right|^2 + (1-t) ||y_2||^2.
\end{equation}
	 
Let us assume for the moment that both \eqref{split1} and \eqref{split2} hold true. By adding up these two inequalities we obtain the following estimate for the left-hand side of \eqref{readytosplit}
\begin{equation} \label{done}
\text{LHS} \eqref{readytosplit} \leq t(||x_1||^2 + ||y_1||^2) + (1-t)(||x_2||^2 + ||y_2||^2) = r^2,
\end{equation}
proving the inequality \eqref{readytosplit} which also concludes the proof of \eqref{new-euclidean}.

 Let us also note that in order to have equality in \eqref{readytosplit} (or \eqref{new-euclidean}) we need that equalities should hold true throughout \eqref{split1} and also \eqref{split2}.

In what follows we shall prove both \eqref{split1} and \eqref{split2} and also analyse the equality cases. In doing so, relations \eqref{P1}, \eqref{P2} and \eqref{identity} will be used in an essential way. Let us start with the first inequality in \eqref{split1}. (The second inequality in \eqref{split1} follows trivially from Jensen's inequality applied to the convex function $x \to ||x||^2$.)
We can write the first inequality in \eqref{split1} as 
 $$  t(1 -t) ( 1+ t(1-t)) ||P_1(x_1, x_2)||^2 \leq (t||x_1|| + (1-t)||x_2||)^2-||tx_1 + (1-t)x_2||^2.$$
 
 Applying \eqref{identity} to $u= tx_1$ and $v=(1-t) x_2$ we see that the above inequality is equivalent to 
 $$  ( 1+ t(1-t)) ||P_1(x_1, x_2)||^2\leq  ||x_1|| \cdot ||x_2||\cdot \Big\| \frac{x_1}{||x_1||} -  \frac{x_2}{||x_2||}\Big\|^2.$$
 
 Using \eqref{P1} we can estimate
 $$ ||P_1(x_1, x_2)||^2 \leq C_1^2 ||x_1||^2 \cdot||x_2||^2 \cdot\Big\| \frac{x_1}{||x_1||} -  \frac{x_2}{||x_2||}\Big\|^2.$$
 
 Now choose $r< \frac{2}{\sqrt{5} C_1}$. Then for all $t\in (0,1) $ we have 
 $$ (1 + t(1-t)) C_1^2 r^2 < 1.$$
 Since $||x_1|| \leq r$  and $||x_2|| \leq r$ this implies 
 $$  (1 + t(1-t)) C_1^2 ||x_1||\cdot ||x_2|| < 1, $$
 and thus the first part of the inequality \eqref{split1}. 
 
 Note that the necessary and sufficient condition to have equality in this inequality is 
 \begin{equation} \label{eq1}
 \frac{x_1}{||x_1||} -  \frac{x_2}{||x_2||}= 0.
\end{equation}  
On the other hand, the necessary and sufficient condition for the equality in the second inequality of 
\eqref{split1} is $||x_1|| = ||x_2||$. We conclude, that the necessary and sufficient condition for equalities in \eqref{split1} is $x_1 = x_2$. 

We can turn now to the proof of \eqref{split2}. Here the estimate \eqref{P2} will be used in an essential manner. We note, that combining the triangle inequality, \eqref{P2} and inequalities $||\delta_t y_1|| \leq t^2 ||y_1||$ and $\|\delta_{1-t}y_2|| \leq (1-t)^2 ||y_2||$ we obtain the following estimate for $||w||$:
$$
||w|| \leq t^2 ||y_1|| + (1-t)^2 ||y_2|| + t(1-t) C_2 (||x_1||\cdot ||y_2|| +  ||y_1||\cdot ||x_2|| +  ||y_1||\cdot ||y_2||).
$$
Choosing $r< \frac{1}{4 C_2}$ and using $||x_1|| \leq r$, $||x_2|| \leq r$ and also  $||y_1|| \leq r$, $||y_2|| \leq r$
we obtain 
\begin{equation*} 
\begin{split}
||w|| &\leq t^2 ||y_1|| + (1-t)^2 ||y_2|| + \frac{t(1-t)}{2}(||y_1|| + ||y_2||)\\ 
&= t ||y_1|| + (1-t) ||y_2|| - \frac{t(1-t)}{2}(||y_1|| + ||y_2||).
\end{split}
\end{equation*}
This inequality can also be written as 
\begin{equation} \label{w-estimate}
||w|| + \frac{t(1-t)}{2}(||y_1|| +||y_2||) \leq t ||y_1|| + (1-t) ||y_2||
\end{equation}  
Observe, that in particular, this implies that 
$$||w|| \leq ||y_1|| + ||y_2||.$$
Thus we infer that 

\begin{align*}
\begin{split}
||w&||^2 (1 + t(1-t))  = ||w||^2 + t(1-t) ||w||^2 \leq ||w||^2 + t(1-t) ||w|| (||y_1|| +||y_2||) \\
&\leq ||w||^2 + t(1-t) ||w|| (||y_1|| +||y_2||) + \left(\frac{t(1-t)}{2}(||y_1|| +||y_2||)\right)^2 \\
&=\left( ||w|| + \frac{t(1-t)}{2}(||y_1|| +||y_2||)\right)^2 \leq (t ||y_1|| + (1-t) ||y_2||)^2.
\end{split}\end{align*}
Note that in the last inequality of the above chain we made use of \eqref{w-estimate}. This proves the second inequality \eqref{split2}. Let us observe that in order to have equality in \eqref{split2} it is necessary to have equalities all throughout in the above chain of estimates. In particular, we must have that 
$$\frac{t(1-t)}{2}(||y_1|| +||y_2||)]=0,$$
as the square of this term has been added to obtain the second inequality. But the only way that this can happen is that $y_1=y_2=0$.

Let us now recall, that the two necessary conditions that we obtained for the equality cases were $x_1=x_2$ and $y_1=y_2=0$. Denoting by $x=x_1=x_2$ this gives that in the equality case we must have $p_1= p_2=(x,0)$ where 
$||x||^2=r^2$ as required.  
This proves the strict horizontal convexity of the norm $N_{HS}$. 

So we obtained that the condition $p_1= p_2=(x,0)$ is a necessary and sufficient condition for the equality in the triangle inequality. If we translate this condition for the starting points $q_1, q_2$ we see that the necessary and sufficient condition for the 
equality 
$$ N_{HS}(q_1\cdot q_2) = N_{HS}(q_1) + N_{HS}(q_2),$$
 is that $q_1=(s_1x,0)$ and $q_2= (s_2x,0)$ for some $s_1, s_2 >0$. That is in fact the original definition of the horizontal strict convexity of $N_{HS}$. 
 
  Let us finally mention that from the proof it is seen that the value of $r_0>0$ in the statement of the theorem can be taken to be 
$$
r_0 = \min \{ 1, \frac{2}{\sqrt{5} C_1}, \frac{1}{4 C_2} \}.
$$
\end{proof}

Using Proposition 6.2 from \cite{H} we obtain as a by-product, the following consequence of Theorem \ref{T:HSCN}:

\begin{corollary} Let $(\mathbf{H}, d_{\mathbf{H}}) $ be a Carnot group with an arbitrary left-invariant homogeneous metric and $(\bG, d_{N_{HS}}) $ be a Carnot group endowed with the left-invariant metric induced by the Hebisch-Sikora norm $N_{HS}$ defined by \eqref{HS} for some $r<r_0$ and hence being horizontally strictly convex. Then all isometric embeddings $\phi: (\mathbf{H}, d_{\mathbf{H}}) \to (\bG, d_{N_{HS}}) $ are affine, i.e. they are compositions of left translations and group homomorphisms. 
\end{corollary}

\section*{Acknowledgements}
This work was initiated at the thematic semester  ``Optimal transport on quantum structures'' (Fall 2022) at the Erd\H{o}s Center, R\'enyi Institute, Budapest. Zolt\'an M. Balogh would like to thank the R\'enyi Institute for the invitation to participate in this event and for the kind hospitality. 
We are grateful to Nicolas Juillet and Katrin F\"assler for their inputs and for several discussions on the topic. We thank the anonymous reviewer for his/her detailed comments and insightful suggestions which improved both the content and the presentation of this paper.

\section*{Declarations}
\paragraph*{{\bf Ethical Approval}} Not applicable.

\paragraph*{{\bf Funding}}
Z. M. Balogh is supported by the Swiss National Science Foundation, Grant Nr. {200020\_191978} and {200021-228012}. T. Titkos is supported by the Hungarian National Research, Development and Innovation Office (NKFIH) under grant agreements no. K134944 and no. Excellence\_151232, and by the Momentum program of the Hungarian Academy of Sciences under grant agreement no. LP2021-15/2021. D. Virosztek is supported by the Momentum program of the Hungarian Academy of Sciences under grant agreement no. LP2021-15/2021, by the Hungarian National Research, Development and Innovation Office (NKFIH) under grant agreement no. Excellence\_151232, and partially supported by the ERC Synergy Grant No. 810115.

\paragraph*{{\bf Availability of data and materials}}
Not applicable.

\end{document}